\documentclass{amsart}

\usepackage{graphicx}
\usepackage[all]{xy}
\usepackage[ansinew]{inputenc}
\usepackage{amsmath}
\usepackage{amscd}
\usepackage{amssymb}
\usepackage{latexsym}
\usepackage[active]{srcltx}
\usepackage{mathrsfs}
\usepackage{color}

\usepackage{xspace}

\usepackage[colorlinks=false, pdfborder={0 0 0}]{hyperref}

\newtheorem{thm}{Theorem}

\theoremstyle{definition}
\newtheorem{cor}[thm]{Corollary}
\newtheorem{lem}[thm]{Lemma}
\newtheorem{prop}[thm]{Proposition}
\newtheorem{defn}[thm]{Definition}
\newtheorem{bsp}[thm]{Example}
\newtheorem{fact}[thm]{Fact}
\newtheorem{conjecture}[thm]{Conjecture}

\newtheorem*{thmA}{Theorem A}
\newtheorem*{thmB}{Theorem B}
\theoremstyle{remark}
\newtheorem{claim}{Claim}

\numberwithin{equation}{section}

\newcommand*{\abs}[1]{\left\vert#1\right\vert}
\newcommand*{\set}[1]{\left\{#1\right\}}
\newcommand*{\pair}[1]{\langle#1\rangle}

\newcommand{\eps}{\varepsilon}

\newcommand*{\cl}[1]{\textbf{cl}_{#1}}

\newcommand{\dist}{\operatorname{dist}}

\newcommand{\N}{\mathbb{N}}

\newcommand{\Q}{\mathbb{Q}}
\newcommand{\R}{\mathbb{R}}

\DeclareMathOperator{\lexmin}{lex\,min }

\newcommand{\K}{\mathbb{K}}

\def \<{\langle}
\def \>{\rangle}
\def \hat {\widehat}
\def \((  {(\!(}
\def \)) {)\!)}

\def \rem {\text{rem}}

\def \cl {\mathrm{cl}}

\newcommand{\Cone}{\mathcal C^1}
\newcommand{\Pa}[1]{\bigl( #1 \bigr)}
\newcommand{\pn}{at most pseudo-enumerable\xspace}
\newcommand{\lrf}{\lambda_r f}
\newcommand{\llf}{\lambda_\ell f}
\newcommand{\Lrf}{\Lambda_r f}
\newcommand{\Llf}{\Lambda_\ell f}
\newcommand{\Disc}{\mathcal D}
\newcommand*{\structure}[1]{\langle #1 \rangle}
\DeclareMathOperator{\fr}{fr}
\newcommand*{\interior}[1]{\mathring{#1}}
\newcommand*{\rest}[1]{\upharpoonright_{#1}}

\begin{document}

\title[Dichotomy for DC fields]{A fundamental dichotomy for definably complete expansions of
  ordered fields}%
\author{Antongiulio Fornasiero and Philipp Hieronymi}

\address
{Seconda Università di Napoli\\
Viale Lincoln 5\\
81100 Caserta\\
 Italy}
\email{antongiulio.fornasiero@gmail.com}
\urladdr{http://www.dm.unipi.it/\textasciitilde fornasiero}

\address{University of Illinois at Urbana-Champaign\\
Department of Mathematics\\
1409 W. Green Street\\
Urbana, IL 61801\\
USA}
\email{phierony@illinois.edu}

\thanks{ A version of this paper will appear in the \emph{Journal of Symbolic Logic.}
The first author was supported by Italian FIRB 2010 "New advances in the Model Theory of exponentiation".
The second author was partially supported by NSF grant DMS-1300402 and by UIUC Campus Research Board award 13086.}


\date{\today}

\maketitle

\begin{abstract} An expansion of a definably complete field either defines a discrete subring, or the image of every definable discrete set under every definable map is nowhere dense.
As an application we show a definable version of Lebesgue’s differentiation theorem.
\end{abstract}

\section{Introduction}

Let $\K$ be an expansion of an ordered field $\<K,<,+,\cdot\>$.
We say $\K$ is \textbf{definably complete} if every bounded subset of $K$
definable in $\K$ has a supremum in $K$. Such structures were first studied by
Miller in \cite{ivp}.
A definably complete expansion of ordered field is
always real closed. The topology considered here is the usual order topology
on $K$ and the product topology on~$K^n$; all rings are taken with~$1$.

The following dichotomy is the main result of the paper.

\begin{thmA} Let $\K$ be definably complete. Then either
\begin{itemize}
\item[(I)] $f(D)$ is nowhere dense for every definable discrete set $D\subseteq K^n$ and every definable function $f: K^n \to K$, or
\item [(II)] $\K$ defines a discrete subring.
\end{itemize}
\end{thmA}

This result is a generalization of \cite[Theorem 1.1]{discrete} from
expansions of the real field to arbitrary definably complete expansions of
ordered fields. The two cases in Theorem~A are indeed exclusive. It is easy to
check that a definable  subring has to be unbounded and that its set
of quotients is dense in~$K$. By definable completeness, the positive elements
of a definable discrete subring of $\K$ form a model of first-order Peano
arithmetic; in \S\ref{sec:Peano} we will see that they even form a model of
second-order Peano arithmetic (seen as a first-order theory).
Hence Theorem~A
separates the class of definably complete expansions of ordered fields into
two very distinct categories.

\medskip

\noindent The significance of Theorem~A comes from its use as a tool to prove
statements about arbitrary definably complete expansions of ordered fields. In
order to show that a statement holds for all such structures, it is now enough
to consider structures having either property (I) or (II) from Theorem A. In
the case when a discrete subring is definable, proofs from second-order
arithmetic often transfer easily to these structures.
On the other hand, if a structure satisfies property
(I), techniques and ideas from the study of o-minimality and related tameness notions can sometimes be applied. As an application of this new proof strategy we present the following definable analogue of Lebesgue’s differentiation theorem, answering a question of Miller from \cite{ivp}.

\begin{thmB}
Let $\K$ be definably complete and let $f: K \to K$ be definable and monotone.
Then $f$ is differentiable on a dense subset of $K$.
\end{thmB}

\subsection*{Notation}  For the rest of the paper, let $\K$ denote a definable
complete expansion of an ordered field $\<K,<,+,\cdot\>$. We say a set is
definable if it is definable in $\K$ with parameters from~$K$.
$\pair{a,b}$ is the ordered pair with elements $a$ and~$b$.
Given a subset $X$ of
$K^{n}\times K^m$ and $a \in K^n$, we denote the set $\{b : \pair{a,b} \in X\}$
by~$X_a$.
As said before, all rings are taken with~$1$.

\subsection*{Acknowledgements} The authors would like to thank Lou van den Dries and the anonymous referees for closely reading the paper and for their valuable comments.

\section{Facts about definable complete fields}

In this section we recall several facts about definably complete expansions of ordered fields. For more details and background, see \cite{ivp}. The following fact is immediate from definable completeness.

\begin{fact}\label{unbounded} Let $Y \subseteq K$ be non-empty closed and definable. Then $Y$ contains a minimum (a maximum) iff $Y$ is bounded from below (from above).
\end{fact}

\begin{fact}[$\textrm{\cite[Lemma 1.9]{ivp}}$]\label{closedunion} Let $Y\subseteq K^2$ be definable such that $Y_a$ is closed and bounded and $Y_a \supseteq Y_b\neq \emptyset$ for every $a,b \in K$ with $a<b$. Then $\bigcap_{a\in K} Y_a\neq \emptyset$.
\end{fact}

\begin{defn}\label{def:successor}
Let $D \subseteq K$ be definable, closed and discrete and let $d \in D$. If $d$ is not the maximum of $D$, we say the minimum of $D_{>d}$ is the \textbf{successor of $d$ in $D$}, written $s_D(d)$.
\end{defn}
\noindent Note that the minimum in the previous definition exists by Fact \ref{unbounded}.

\begin{fact}\label{minmax} Let $D \subseteq K$ be definable, closed and
discrete.
If $D$ has a minimum (a maximum), so has every definable subset of $D$.
\end{fact}

\begin{defn} A subset $A \subseteq K^n$ is called \textbf{pseudo-finite} if it is definable, closed, bounded and discrete. We call $A$ \textbf{at most pseudo-enumerable} if there exists a definable closed discrete set
$D \subset K_{\geq 0}$ and a definable function $f: D \to K^n$ such that $f(D) = A$.
\end{defn}

\noindent The notion of a pseudo-finite set was introduced in \cite{local} and the notion of \pn in \cite{enum}.

\begin{fact}[$\textrm{\cite[Main Theorem]{enum}}$]\label{nointerior}
If $A \subseteq K^n$ is \pn, then it has no interior.
\end{fact}

\begin{fact}[$\textrm{\cite[Lemma 2.22]{local}}$]\label{imagepf} Let $D\subseteq K^n$ be pseudo-finite and let $f : D \to K^m$ be a definable function. Then $f(D)$ is pseudo-finite. In particular, $f$ achieves a minimum and a maximum on $D$.
\end{fact}

\begin{fact}[$\textrm{\cite[Lemma 4.14]{enum}}$]\label{pseudo} Every definable
discrete subset of $K^n$ is \pn.
\end{fact}

\noindent Fact \ref{pseudo} simplifies our task to prove Theorem A
considerably.
To establish Theorem~A, it is now enough to show that whenever $\K$ defines a
closed and discrete set $D\subseteq K_{\geq 0}$ and a function $f:D \to K$
with $f(D)$ somewhere dense, then $\K$ defines a discrete subring.

\begin{defn}
A definable family $(X_t: t \in D)$ is \textbf{\pn} if its index set $D$ is \pn.
\end{defn}

\noindent The following fact was implicitly proved in \cite{enum}. For the reader's convenience, we have included a proof here.

\begin{fact}\label{lem:discrete-union}
\begin{itemize}
\item[(1)] The union of an \pn family of discrete sets is \pn.
\item[(2)] Let $(X_t: t \in K)$ be a definable increasing family of discrete
subsets of $K^n$.  Then $\bigcup_{t\in K} X_t$ is \pn.
\end{itemize}
\end{fact}
\begin{proof} Statement (1) is \cite[Corollary 4.16]{enum}.
We now consider (2). By \cite[Theorem 3.3]{local}, $\K$ either defines a discrete, closed
and unbounded set, or every discrete set definable in $\K$ is pseudo-finite.
We now handle the cases separately. If every discrete definable set in $\K$ is pseudo-finite, then each $X_t$ is pseudo-finite. By \cite[Theorem 3.3]{local}
$\bigcup_{t\in K} X_t$ itself is pseudo-finite. Now suppose that there exists $D \subseteq K_{\geq 0}$ definable, discrete, closed
and unbounded. Since $(X_t: t \in K)$ is increasing and
$D$ is unbounded, $\bigcup_{t\in K} X_t = \bigcup_{t \in D} X_t$. By (1) applied to the family~ $(X_t: t \in D)$, $\bigcup_{t\in K} X_t$ is \pn.
\end{proof}



\section{Natural fragments and asymptotic extraction}

In this section we generalize the idea of asymptotic extraction, first introduced by Miller in \cite[p. 1484]{proj}, to definably complete fields. Since the original approach is not strong enough to yield the desired results, we adjust the method developed in \cite[Lemma 1]{countable} to extract larger and larger fragments of the natural numbers.
\begin{defn} Let $D$  be a definable, closed and discrete subset of $K_{\geq 0}$. We say
that $D$ has step $1$ if, for every $d \in D$ with $d \neq \max(D)$,
$s_D(d) = d+ 1$. We say that $D$ is a
\textbf{natural fragment} if it is either empty, or if $D$ has step 1 and  $0 \in D$.
\end{defn}

\begin{lem}\label{nfcomp} Let $D$ and $E$ be natural fragments. Then either $D\subseteq E$ or $E\subseteq D$.
\end{lem}
\begin{proof} Suppose not. Let $d = \min(D \setminus E \cup E \setminus D)$. Without loss of generality, assume $d \in D$. Since $0 \in D \cap E$, $d>0$. Since $D$ is a natural fragment, $d-1 \in D$. Since $d$ was chosen to be minimal, $d-1 \in E$ as well. Since $d \notin E$, $d-1$ has to be the maximum of $E$. Since $D \cap [0,d-1]= E \cap[0,d-1]$ by minimality of $d$, we have $E \subseteq D$.
\end{proof}

\begin{cor} \label{nfcor} Let $(X_t : t \in I)$ be a definable family such that $X_t$ is a natural fragment for each $t \in I$. Then $\bigcup_{t\in I} X_t$ is a natural fragment.
\end{cor}

\noindent It is worth noting that by Lemma \ref{nfcomp} the union of all natural fragments, although not necessarily definable, is closed, discrete, contains $0$ and has step $1$.

\begin{defn} Let $D$ be a definable, closed and discrete subset of $K_{\geq 0}$ and $\varepsilon \in K_{>0}$. We say that $D$ is an \textbf{$\varepsilon$-natural fragment} if
\begin{itemize}
\item [(1)] $|s_{D}(d) - (d + 1)| < \varepsilon$ for every $d \in D$ with $d\neq \max D$,
\item [(2)] $\dist(D,0) < \varepsilon$.
\end{itemize}
For $a \in K_{\geq 0}$, we say $D$ is an \textbf{$\varepsilon$-natural fragment close to $a$} if $\dist(D,a) <\varepsilon$.
\end{defn}

\noindent The next Lemma shows that the property of being an $\varepsilon$-natural fragment for some $\varepsilon$ is preserved under small changes.

\begin{lem}\label{enatfrag} Let $\varepsilon \in K_{>0}$ with $\varepsilon< \frac{1}{4}$, let $D$ be a $\varepsilon$-natural fragment close to $a$ and let $f: D \to (-\varepsilon,\varepsilon)$ be a definable function. Then
$$
E:=\{ d + f(d) \ : \ d \in D\}
$$
is a $3\varepsilon$-natural fragment close to $a$.
\end{lem}
\begin{proof} Set $g(d) := d+f(d)$ for $d \in D$. It is immediate that (2)
holds for $E$ and $3\varepsilon$, since it holds for $D$ and
$\varepsilon$. Since (1) holds for $D$ and $\varepsilon$ and $\varepsilon <
1/4$, $g(s_{D}(d)) = s_{E}(g(d))$ for every $d\in D$ with $d \neq \max D$. Moreover, for every $d
\in D$ with $d \neq \max D$,
\[
\abs{s_{E}(g(d)) - g(d) - 1}  < 2\varepsilon + \abs{s_D(d)-d-1} < 3\varepsilon.
\]
Hence (1) holds for $E$ and $3\varepsilon$. Hence $E$ is a
$3\varepsilon$-natural fragment close to~$a$.
\end{proof}

\begin{defn} Let $(Y_t : t\in I)$ be a definable family of subsets of~$K$.
The \textbf{natural fragment extracted from $(Y_t : t\in I)$} is the set of $d
\in K_{\geq 0}$ such that for every
$\varepsilon\in K_{>0}$ there exists $t \in I$ such that $Y_t$ is an $\varepsilon$-natural fragment close to~$d$.
\end{defn}

\noindent It is not obvious that the object defined in the previous definition is a natural fragment in sense defined before. The following Lemma establishes that this is indeed the case.

\begin{lem} Let $(Y_t : t\in I)$ be a definable family of subsets of~$K$. Then the natural fragment extracted from $(Y_t : t\in I)$ is a natural fragment.
\end{lem}
\begin{proof} Let $D$ be the natural fragment extracted from $(Y_t : t\in I)$. Since the empty set is a natural fragment, we reduce to the case that $D$ is non-empty. It follows easily from the definitions that $0\in D$ whenever $D$ is non-empty.\\

\noindent For $d \in D$ consider the definable set $E_d$ consisting of all $e \in K$ with $e\leq d$ such that for every $\varepsilon\in K_{>0}$ there exists $t \in I$ such that $Y_t$ is an $\varepsilon$-natural fragment close to $d$ and $\dist(Y_t,e)< \varepsilon$. Note that $d \in E_d$ and $E_d \subseteq D$. Hence $\bigcup_{d\in D} E_d = D$. Thus by Corollary \ref{nfcor} it is enough to show that each $E_d$ is a natural fragment.\\

\noindent Let $d \in D$. We first show that $e+1 \in E_d$ for every $e \in
E_d$ with $e\leq d-1$. Let $\varepsilon \in K$ such that $0 < \varepsilon < 1$. Take $t \in I$ such that $Y_t$ is a $\frac{\varepsilon}{2}$-natural fragment close to $d$ and $\dist(Y_t,e)<\frac{\varepsilon}{2}$. Let $y \in Y_t$ be such that $|e-y|<\frac{\varepsilon}{2}$. Since $e \leq d-1$ and $\dist(Y_t,d)<\frac{\varepsilon}{2}$, $y$ is not the maximum of $Y_t$. Then
\begin{align*}
|s_{Y_t}(y) - (e+1)| & = |s_{Y_t}(y) +y - y- (e+1)|\\
 &\leq  |s_{Y_t}(y) -y-1| + |y - e| < \varepsilon.
\end{align*}
Hence $\dist(Y_t,e+1)<\varepsilon$. Thus $e+1 \in E_d$. Similarly, we can show that $e-1 \in E_d$ for every $e \in E_{d}$ with $e\geq 1$.\\

\noindent Consider
\[
B := \{ \ e \in E_d \ : \ [e,e+1) \cap E_d = \{e\} \ \}.
\]
Note that $B$ is closed and discrete and $d \in B$. We will now show that $B$
is a natural fragment. It is easy to see that $0 \in B$. Let $e \in B$ and
suppose $e \leq d-1$. Then $e+1 \in E_d$. Towards a contradiction assume $e+1
\notin B$.
Then there is $l\in E_d$ such that $e+1 < l < e+2$.
Since $l \geq 1$ and $l \in E_d$,
we have $l-1 \in E_d$ with $e < l-1 <e+1$. Hence $e \notin B$, a contradiction.\\

\noindent It is left to show that $E_d=B$. Towards a contradiction suppose there is $e \in E_d \setminus B$. By Fact \ref{minmax} there is a maximal $l \in B$ smaller than $e$. Since $l<d$, $l+1 \in B$. Since $l \in B$ and $e \notin B$, $l+1<e$. A contradiction against the maximality of~$l$. Hence $E_d = B$.
\end{proof}

It is worth pointing out that until this point only the additive structure of $\K$ has been used.

\begin{prop}\label{nat1} Let $D$ be an unbounded natural fragment. Then
$\<D,+,\cdot,<\>$ is a model of first-order Peano arithmetic.
Moreover, $D \cup -D$ is a definable discrete subring of~$\K$.
\end{prop}
\begin{proof} Let $D$ be an unbounded natural fragment. We first show that $D$
is closed under addition. Suppose not. By Fact \ref{minmax} we can take $d\in
D$ minimal such that there is $e \in D$ with $d+e \notin D$.
Clearly, $d \neq 0$.
Since $d$ is minimal, $(d-1)+(e+1) \in D$. A contradiction. Hence $D$ is
closed under addition.

Now suppose that $D$ is not closed under multiplication.
Again take $d\in D$ minimal such that there is $e \in D$ with $d\cdot e \notin
D$. Clearly, $d \neq 0$.
By minimality of~$d$, $(d-1)\cdot e \in D$ and hence $(d-1)\cdot e +e \in
D$. Hence $D$ is closed under multiplication. Since every definable subset of
$D$ has a minimum by Fact \ref{minmax}, $\<D,+,\cdot,<\>$ satisfies the
first-order induction axiom. Hence $\<D,+,\cdot,<\>$ is a model of first-order
Peano arithmetic. 

Now set $Z := D \cup -D$. It follows immediately that $\<Z,+,\cdot\>$ is a discrete subring of~$\K$.
\end{proof}

\section{Best approximations and the proof of Theorem A}

Let $\K$ be a definably complete expansion of an ordered field that defines a closed and discrete set $D\subseteq K_{\geq 0}$ and a function $f:D \to K$ with $f(D)$ somewhere dense. In order to establish Theorem A, it is enough by
Fact \ref{pseudo} to define a discrete subring. By Proposition \ref{nat1} it suffices to define an unbounded natural fragment. After composing $f$ with a semialgebraic function we can assume that $f(D)$ is dense in $(0,1)$. First several definitions related to this function $f$ will be introduced. These definitions were first used  for expansions of $\R$ by Hieronymi and Tychonievich in~\cite{group}.

\begin{defn} Let $c\in (0,1)$. We say $d\in D$ is a \textbf{best approximation of $c$ from the left} if $f(d) < c$ and
$$
f(D_{<d}) \cap \big(f(d),c\big) = \emptyset.
$$
We write $L_c$ for the set of best approximations of $c$ from the left. Similarly, we say $d \in D$ is a \textbf{best approximation of $c$ from the right} if $f(d)>c$ and
$$
f(D_{<d}) \cap \big(c,f(d)\big) = \emptyset.
$$
and write $R_c$ for the set of best approximations of $c$ from the right.\newline
For $d \in D$, we write
$$
L_{c,d} := L_c \cap [0,d] \hbox{ and } R_{c,d} := R_c \cap [0,d],
$$
and
$$
l_{c,d} := \left\{
             \begin{array}{ll}
               f(\max L_{c,d}), & \hbox{if $L_{c,d} \neq \emptyset$;} \\
               0, & \hbox{otherwise,}
             \end{array}
           \right.
 \hbox{ and } r_{c,d} := \left\{
             \begin{array}{ll}
               f(\max R_{c,d}), & \hbox{if $R_{c,d} \neq \emptyset$;} \\
               1, & \hbox{otherwise.}
             \end{array}
           \right.
$$
\end{defn}
\noindent Since $D$ is closed and discrete, both $L_c$ and $R_c$ are closed
and discrete by Fact~\ref{minmax}.  Since $D_{\leq d}$ is pseudo-finite, so is
$f(D_{\leq d})$ by Fact~\ref{imagepf}. Hence both $L_c$ and $R_c$ are
non-empty. It is easy to check that by density of $f(D)$ both $L_c$ and $R_c$
are unbounded and $c = \sup f(L_c) = \inf f(R_c)$ for $c\in (0,1)$. Moreover $L_{c,d}$ and $R_{c,d}$ are pseudo-finite and the maximum used in the above definition actually exists. It also worth pointing out that this implies $l_{c,d} < c < r_{c,d}$.\\

\begin{lem}\label{sameapprox} Let $a \in (0,1) \setminus f(D)$ and $d\in D$. Then $L_{a,d} = L_{b,d}$ and $R_{a,d} = R_{b,d}$ for every $b \in \big(l_{a,d},r_{a,d}\big)$.
\end{lem}
\begin{proof} By definition of $l_{a,d}$ and $r_{a,d}$,
\[
f(D_{\leq d}) \cap (\big(l_{a,d},a\big) \cup \big(a,r_{a,d}\big)) = \emptyset.
\]
Since $a \notin f(D)$, $f(D_{\leq d}) \cap \big(l_{a,d},r_{a,d}\big) = \emptyset$. Hence for all $b \in \big(l_{a,d},r_{a,d}\big)$
\[
\{ e \in D_{\leq d} \ : \ f(e) < b\} = \{ e \in D_{\leq d} \ : \ f(e) < a\}
\]
and
\[
\{ e \in D_{\leq d} \ : \ f(e) > b\} = \{ e \in D_{\leq d} \ : \ f(e) > a\}.
\]
Thus $L_{a,d} = L_{b,d}$ and $R_{a,d} = R_{b,d}$.
\end{proof}

The strategy for the rest of proof is as follows. We will introduce a
definable family using the notions introduced above. Then it will be shown
that the natural fragment extracted from this family is unbounded. The idea
how to show the last statement is the following: suppose there is $b\in K$,
$\varepsilon \in K_{>0}$ and a suitable semialgebraic function $g$ such that
the image of a definable subset of $L_{b,d} \times \{b\} \times R_{b,d}$ is an
$\varepsilon$-natural fragment close to some $n\in K$.
By Lemma \ref{sameapprox}, the set $L_{b,d}$ and $R_{b,d}$ do not change on an
interval around $b$. Being careful with the definitions we will use this
statement to show that we can find an element $c$ close to $b$ and $d' \in D$ such that the image of a definable subset of $L_{c,d'} \times \{c\} \times R_{c,d'}$ under $g$ is a $6\varepsilon$-natural fragment close to $n+1$.\\

\noindent Let $g: K^3 \to K$ be
\[
g(a,b,c) := \left\{
                  \begin{array}{ll}
                    \frac{c - a}{b - a} & \hbox{if } a < b < c, \\
                    0 & \hbox{otherwise.}
                  \end{array}
                \right.
\]

\noindent We will now define a family of definable sets from which we extract an unbounded natural fragment. Let $\pair{a,b} \in (0,1)^2$ and $d\in D$. Define
\[
Y_{a,b,d} := \{ 0 \} \cup \{g(l_{b,e},b,r_{b,e})  \ : \ e \in L_{a,d} \}.
\]
Let $J$ be the set of $\pair{a,b,d} \in ((0,1)\setminus f(D))^2 \times D$ such that the map $e \mapsto g(l_{b,e},b,r_{b,e})$ is strictly increasing on $L_{a,d}$.
Note that $(Y_{a,b,d})_{\pair{a,b,d} \in J}$ is indeed a definable family.

\begin{lem}\label{mainprooflem} Let $\pair{a,b,d} \in J$, $c\in (0,1)$, $u \in K$ and $\varepsilon \in K_{>0}$ with $\varepsilon < \frac{1}{4}$. Then
\begin{itemize}
\item [(i)] if $L_{a,d} = L_{c,d}$, then $Y_{a,b,d} = Y_{c,b,d}$.
\item [(ii)] if $Y_{a,b,d}$ is an $\varepsilon$-natural fragment close to $u$,
then there is an interval $I$ around $b$ such that for all $c\in I\setminus
f(D)$, we have $\pair{a,c,d} \in J$ and
$Y_{a,c,d}$ is a $3\varepsilon$-natural fragment close to $u$.
\end{itemize}
\end{lem}
\begin{proof} Statement (i) is immediate from the definitions. For (ii) let $I_0$ be the interval $(l_{b,d},r_{b,d})$. By Lemma \ref{sameapprox} and $b\notin f(D)$, $L_{c,d} = L_{b,d}$ and $R_{c,d} = R_{b,d}$ for every $c \in I_0$. For each $e \in L_{a,d}$ let $I_e$ be the maximal open subinterval of $I_0$ containing $b$ such that for each $c \in I_e$
\begin{equation}\label{eq1}
|g(l_{c,e},c,r_{c,e}) - g(l_{b,e},b,r_{b,e})| < \varepsilon.
\end{equation}
This choice is possible, since $g$ is continuous in the second coordinate and $l_{c,e}=l_{b,e}$ and $r_{c,e}=r_{b,e}$ for every $c \in I_0$.
Note that the maps $e \in L_{a,d} \mapsto \sup I_e$ and $e \in L_{a,d} \mapsto \inf I_e$ are definable. Hence by Fact~\ref{imagepf} both functions have a maximum and a minimum on $L_{a,d}$. Hence there is $e_1,e_2 \in L_{a,d}$ such that
$$
\bigcap_{e \in L_{a,d}} I_e = \big (\inf I_{e_1}, \sup I_{e_2}\big ).
$$
Let $I$ be this open interval. Since $b \in I$, $I$ is non-empty.
Since $\varepsilon < \frac{1}{4}$
and $Y_{a,b,d}$ is an $\varepsilon$-natural fragment,
the map $e \mapsto g(l_{c,e},c,r_{c,e})$ is strictly increasing on $L_{a,d}$
for every $c \in I\setminus f(D)$.
Hence  $\pair{a,c,d} \in J$ for all such $c$.\\

\noindent Let $c \in I$. Let $k: Y_{a,b,d}  \to (-\varepsilon,\varepsilon)$
map $0$ to $0$ and $g(l_{b,e},b,r_{b,e})$ to
\[
g(l_{c,e},c,r_{c,e}) - g(l_{b,e},b,r_{b,e}).
\]
This function is well-defined, since $\pair{a,b,d} \in J$ and hence $e \mapsto g(l_{b,e},b,r_{b,e})$ is strictly increasing on $L_{a,d}$. By definition
\[
Y_{a,c,d} = \{ y + k(y) \ : \ y \in Y_{a,b,d} \}.
\]
By definability of $k$, \eqref{eq1} and Lemma~\ref{enatfrag}, this set is a $3\varepsilon$-natural fragment.
\end{proof}

\begin{thm}\label{thm:unbounded-fragment}
The natural fragment extracted from $(Y_{a,b,d} : \pair{a,b,d} \in J)$ is unbounded.
\end{thm}
\begin{proof}
Let $F$ be the natural fragment extracted from $(Y_{a,b,d} : \pair{a,b,d} \in J)$. We first show that $F$ is non-empty. 
It is enough to find for every $\eps \in K_{>0}$ a triple $\pair{a,b,d} \in J$ such that $Y_{a,b,d}$ is an $\eps$-natural fragment up to~$1$.
Let $d\in D$ be the smallest element of $D$. Take $a, b \in [0,1] \setminus f(D)$ such that
\[
0 < \frac{f(d)}{1 + \eps} <b < f(d) < a < 1.
\]
Then $L_{a,d} = \{d\}$, $L_{b,d}=\emptyset$ and $R_{b,d} = \{d\}$. Hence $l_{b,d} = 0$ and $r_{b,d}=f(d)$. Thus
\[
|g(l_{b,d},b,r_{b,d})-1| = |\frac{f(d)}{b} - 1| < \eps.
\]
Hence $Y_{a,b,d} = \{0,\frac{f(d)}{b} \}$ is an $\varepsilon$-natural fragment up
to~$1$.\\

\noindent Now towards a contradiction suppose that $F$ is bounded. Let $n$ be
the maximum of~$F$. We will establish a contradiction against the maximality
of $n$. For this, it is enough to construct for every $\varepsilon \in K_{>0}$
a triple $\pair{a,b,d} \in J$  such that $Y_{a,b,d}$ is an $\varepsilon$-natural
fragment close to $n+1$.\\

\noindent Let $\varepsilon \in K_{>0}$. Since $n$ is in the natural fragment extracted from $(Y_{a,b,d} : \pair{a,b,d} \in J)$, there is $\pair{u,v,e}\in J$ such that $Y_{u,v,e}$ is an $\frac{\varepsilon}{6}$-natural fragment close to~$n$. Let $I$ be the interval around $v$ given by Lemma \ref{mainprooflem}(ii) such that for every $w \in I\setminus f(D)$, $Y_{u,w,e}$ is an $\frac{\varepsilon}{2}$-natural fragment close to $n$ and $\pair{u,w,e} \in J$.\\

\noindent
Let $d_0$ be an element of $D_{\geq e}$ such that there are $e_1,e_2 \in D$
with $e_1,e_2\leq d_0$, $f(e_1) < f(e_2)$, and
\[
\big( f(e_1), f(e_2) \big) \subseteq I.
\]
Such an element exists because of the density of $f(D)$. Choose $a\in
K\setminus f(D)$ such that $l_{u,e} < a$ and
\[
\big(l_{u,e},a\big) \cap f(D_{\leq d_0}) = \emptyset.
\]
We can find such an element because $f(D_{\leq d_0})$ is pseudo-finite and $f(D)$ does not have interior by Fact \ref{nointerior}. Now let $d \in D$ be the smallest element in $D_{\geq d_0}$ with
$$
f(d) \in \big(l_{u,e},a\big).
$$
Then $L_{a,d} = L_{u,e} \cup \{d\}$.\\

\noindent It is left pick to $b$. First take $e_1,e_2 \in D_{\leq d}$ such
that $f(e_1) < f(e_2)$,
\[
\big( f(e_1), f(e_2) \big) \subseteq I \quad \text{and} \quad \big( f(e_1), f(e_2) \big) \cap f(D_{\leq d}) = \emptyset.
\]
This choice is possible, because $f(D_{\leq d})$ is pseudo-finite.  Now pick $b\in \big(f(e_1),f(e_2)\big)$ such that $b \notin f(D)$ and
\begin{equation}\label{eqmainproof1}
| g(f(e_1),b,f(e_2)) - (n+1)| < \frac{\varepsilon}{2}.
\end{equation}
By our choice of $b$, $f(e_1) = l_{b,d}$ and $f(e_2) = r_{b,d}$.
Since $b \in I \setminus f(D)$, $Y_{u,b,e}$ is a $\frac{\varepsilon}{2}$-natural fragment close to $n$ and
$\pair{u,b,e} \in J$. Since $L_{u,e} = L_{a,e}$ by choice of $a$, we have that
$\pair{a,b,e} \in J$ and $Y_{u,b,e} = Y_{a,b,e}$ by Lemma
\ref{mainprooflem}(i).
Hence $Y_{a,b,e}$ is a $\frac{\varepsilon}{2}$-natural fragment close to~$n$. If $Y_{a,b,e}$ is a $\frac{\varepsilon}{2}$-natural fragment close to $n+1$, then $Y_{a,b,e}$ is also a $\varepsilon$-natural fragment close to $n+1$ and hence we are done. Thus from now on we can assume that $Y_{a,b,e}$ is not a $\frac{\varepsilon}{2}$-natural fragment close to $n+1$. Then
\begin{equation}\label{eqmainproof2}
|\max Y_{a,b,e} - n| < \frac{\varepsilon}{2}.
\end{equation}
Now set $z:= g(l_{b,d},b,r_{b,d})$. Since $L_{a,d} = L_{u,e} \cup \{d\}$, we have
\[
Y_{a,b,d}=Y_{a,b,e} \cup \{z\}.
\]
Then by \eqref{eqmainproof1} and \eqref{eqmainproof2}
$$
|z - \max Y_{a,b,e} - 1| \leq |z -(n+1)| + |\max Y_{a,b,e} - n| < \frac{\varepsilon}{2} + \frac{\varepsilon}{2} = \varepsilon.
$$
Hence $Y_{a,b,d}$ is an $\varepsilon$-natural fragment close to $n+1$.
\end{proof}

The proof of the above theorem can be easily adapted to show that, for every
pseudo-finite set $F \subseteq K_{\geq 1}$, $\{0 \} \cup F$~can be approximated
arbitrarily close by some $Y_{a, b,d}$: i.e., for every $\eps > 0$
there exists $\pair{a,b,d} \in J$ such that $\dist(\{ 0 \} \cup F, Y_{a,b,d}) <
\eps$.

As shown in Proposition~\ref{nat1}, if $D$ is the unbounded natural fragment
extracted in Theorem~\ref{thm:unbounded-fragment}, then $D \cup -D$ is a
definable discrete subring of $\K$, and Theorem~A follows.


\section{Unrestrained DC structures}
\label{sec:Peano}
\begin{defn}We say that $\K$ is \textbf{unrestrained} if it defines a discrete subring,
otherwise, we say that $\K$ is \textbf{restrained}.
\end{defn}
We claim that unrestrained structures are the same as model of second-order
arithmetic, in a sense that we will make precise.
The result of this section will be used in \S\ref{sec:thmB}.

First, we make precise what we mean by model of second-order arithmetic;
as a background reference we use \cite{Simpson}, especially its \S1.
Let $L_2 := \structure{N, D; 0, 1, +, \cdot, <}$ be the (first-order!) 2-sorted language of second-order arithmetic,
with a sort $N$ for (``natural'')
numbers (which will be denoted by lowercase letters) and one sort $D$ for
sets (which will be denoted by uppercase letters), with number constants $0$ and~$1$,
binary operations $+$ and~$\cdot$ and a binary
relation $<$ on numbers, and a binary relation $\in$ between
numbers and sets.
Let $L$ be an arbitrary expansion of~$L_2$ (here we differ
from~\cite{Simpson}, where only structures in the language $L_2$ are
considered); notice that we allow  extra function symbols and predicates that involve the sort $D$ and not only new function and relation symbols on~$N$.
A model of second-order  arithmetic is a (again, first-order) $L$-structure
$\mathcal N := \structure{N, D; 0, 1, +,\cdot, < , \dotsc}$
 satisfying the following axioms:
\begin{description}
\item[Basic axioms]  $\structure{N, + ,0 , 1, +, \cdot, <}$ is the positive cone of
a discrete linearly ordered ring~$Z$;
\item[Extension axiom]
\[
\forall X \forall Y\ (\forall n (n \in X \leftrightarrow n \in
Y) \rightarrow X = Y);
\]
\item[Induction axiom]
\[
\forall X\ ((0 \in X \ \&\ \forall n (n \in X \rightarrow n + 1 \in X))
\rightarrow X = N);
\]
\item[Comprehension scheme]
\[
\exists X \forall n\ (n \in X \leftrightarrow \phi(n)),
\]
where $\phi(n)$ is any $L$-formula in which $X$ does not occur freely.
\end{description}

Remember that we view second-order arithmetic as a theory in first-order logic;
thus, the theory will have models besides the standard one, given by $\N$ and
all its subsets.

\bigskip
\subsection{From unrestrained structures to models of arithmetic}
First, 
we show how to transform an unrestrained structure
into a model of second-order arithmetic.

For the rest of this subsection, let $\K$ be unrestrained.
Let $Z$ be a definable discrete subring of~$\K$.
Note that $Z$ is the unique subring with that property: if $Z'$ were
another discrete definable subring of~$K$,
one considers the minimum positive element of $Z \Delta Z'$ and easily reaches a contradiction against $Z\neq Z'$.
We will denote the non-negative elements of $Z$ by $N$ and the fraction field
of~$Z$ by~$Q$. We start by transferring some of the coding techniques, in particular recursion, to our setting. As most of the proofs are direct transfers of the classical ones, we leave most of the details to the reader.

It is already clear that $N$ is a model of first-order arithmetic.

\begin{lem}\label{lem:Goedel}
There is a definable map $\beta : N \times N \to N$
such that for each $l \in N$ and each definable map
$f: N_{\leq l} \to N$ there is $k \in N$ such that $\beta(k,i) = f(i)$ for $i\leq l$.
\end{lem}
\begin{proof}\setcounter{claim}{0}
Since the function $f$ may definable using parameters outside~$N$, we will
remind the proof (we refer to \cite[\S II.2]{Simpson} for the details).
Since $N$ is a model of first-order arithmetic, 
there is a definable bijection $\theta: N \times N \to N$.
Define $\beta'(r, a, i) := \rem(r, (i + 1) \cdot a + 1)$,
where $\rem(x, y)$ denotes the remainder after integer division of $x$ by~$y$.
Let $\beta(k, i) := \beta'(\theta^{-1}(k), i)$.

\noindent Let $l \in N$ and $f: N_{\leq l} \to N$ be definable. It is left to show that there exist $r, a \in N$ such that
$\beta'(r, a ,i) = f(i)$ for $i\leq l$.
Since $N$ is a model of first-order arithmetic, we can find $a \in N$ such that $f(i) < a$ for each $i \leq l$ and all elements of
\[
\{ (i+1)a +1 \ : \ i \in N_{\leq l} \}
\]
are pairwise coprime. We denote $(i+1)a+1$ by $k_i$.
To finish the construction we just need to establish the following claim.

\begin{claim}
For each $m \in N_{\leq l}$, there exists $r \in N$ such that for each $i \in N_{\leq m}$
\[
\rem(r, k_i) = f(i).
\]
\end{claim}
Suppose not. Let $m \in N_{\leq l}$ be minimal such that $r$ as in the claim does not exist.
By minimality of~$m$, there is $r' \in N$ be such that for every
$i \in \N_{\leq m-1}$
\[
\rem(r',k_i) = f(i).
\]
Note that the set $\{ k_i \ : \ i \in N_{<m}\}$ is definable inside~$N$.
Let $r \in N$ such that $\rem(r, k_i) = \rem(r', k_i)$ for each $i < m$,
and $\rem(r, k_m) = f(m)$. Such an $r$ exists by the Chinese Remainder Theorem in $N$. The Chinese Remainder holds in $N$ because $N$ is a model
of first-order arithmetic.
Contradiction.
\end{proof}

From the proof of the above Lemma, it is clear that $\beta$ is already
definable in  $\<N, + ,\cdot, < \>$, and hence for every $l \in N$,
every definable subset of $N_{< l}$ is definable in~$\< N, + ,\cdot, < \>$.

\begin{lem}\label{lem:recursive} Let
$c: K^n \to N$ and $g : K^{n} \times N \to N$ be definable. Then there is
a unique definable function $f : K^n \times N \to N$ such that
for all $a \in K^n$
\begin{align*}
f(a,0) &= c(a),\\
f(a,i+1) &= g(a,f(a,i)).
\end{align*}
\end{lem}
\begin{proof}
As in the real case, given $a \in K^n$ and $j, l \in N$,
we define $f(a, j) = l$ if there exists $k \in N$ such that
\[
\begin{aligned}
\beta(k,0) & = a;\\
\beta(k, j) &= l;\\
\forall i \in N \text{ such that } i < j,\quad \beta(k, i+1) &=
g(a, \beta(k,i)).
\qedhere
\end{aligned}
\]
\end{proof}

\begin{cor}\label{cor:defbij} Let $X \subseteq N$ be unbounded and definable. Then there is definable bijection from $N$ to $X$.
\end{cor}
\begin{proof} Let $f : N \to X$ be the function that takes $0$ to the minimum of $X$ and $i+1$ to the successor of $f(i)$ in $X$. By Lemma \ref{lem:recursive} $f$ is definable.
\end{proof}

\begin{defn}
Let $A$ and $B$ be definable sets.
Let $\Delta$ be a family of functions from $B$ to~$K^m$.
We say that \textbf{$\Delta$ is in definable bijection with $A$} if there exists a
definable map $\alpha: A \times B  \to K^m$ such that:
\begin{enumerate}
\item for every $f \in \Delta$ there exists a unique $a \in A$ such that for
every $b \in K^n$ $f(b) = \alpha(a,b)$;
\item for every $a \in A$ the map $x \mapsto \alpha(a,x)$ is in~$\Delta$.
\end{enumerate}
With the above notation, we denote by $\hat\alpha: \Delta \to A$ the map
sending $f \in \Delta$ to the unique $a \in A$ satisfying (1).

If $\Gamma$ is a family of subsets of $B$, we say that \textbf{$\Gamma$ is in
definable bijection with  $A$} if the family of characteristic functions of the
sets in $\Gamma$ is in definable bijection with~$A$.
By abuse of notation, if $\alpha: A \times B \to \set{0,1}$ is the
corresponding map, we denote by $\hat\alpha: \Gamma \to A$ the map sending
$X \in \Gamma$ to the unique $a \in A$ satisfying the analogue of~(1).
\end{defn}

\begin{bsp} The family of open balls in $K^n$ is in definable bijection with
$K^n \times K_{> 0}$.
\end{bsp}

\begin{lem}\label{lem:defbijfinite} The family of
definable bounded subsets of $N$ is in definable bijection with $N$.
\end{lem}
\begin{proof} Let $C \subseteq N$ be the set of all $k \in N$ such that $\beta(k,i) \in \{0,1\}$ for all $i \in N$.
Define $\gamma : C \times N \times N \to \{0,1\}$ by
\[
(k,l,i) \mapsto \left\{
                  \begin{array}{ll}
                    \beta(k,i), & \hbox{if $i\leq l$;} \\
                    0, & \hbox{otherwise.}
                  \end{array}
                \right.
\]
Since $N$ is a model of first-order arithmetic, there is a definable bijection $\theta: N \times N \to N$. Let $D$ be $\theta(C \times N)$. Now consider the subset $E$ of $D$ containing all $k \in N$ such that there is no $k' \in N$ with $k'<k$ and
\[
\{ i \in N : \gamma(\theta^{-1}(k'),i)=1 \} = \{ i \in N : \gamma(\theta^{-1}(k),i)=1 \}.
\]
By Lemma \ref{lem:Goedel}, for every bounded definable subset $X$ of $N$ there is $k \in D$ such that $\{ i \in N : \gamma(\theta^{-1}(k),i)=1 \} = X$. Hence by Fact \ref{unbounded}, there is a unique $k \in E$ with this property.
Hence the family of definable bounded subsets of $N$ is in definable bijection with $E$. By Corollary \ref{cor:defbij} $E$ is in definable bijection with $N$. Thus the family of definable bounded subsets of $N$ is in definable bijection with $E$.
\end{proof}

\begin{cor}\label{cor:tuples}
 The family of definable subsets of $N$ is in definable bijection with~$K$.
\end{cor}
We denote by $\hat\epsilon$ the corresponding bijection.
\begin{proof}\setcounter{claim}{0}
The idea of the proof is to use the expansion in base $2$ to encode definable
subsets of $N$ as elements of~$K$.
Let $E$ be the family of all definable subsets of~$N$ and
$C$ be the family of unbounded definable subsets of~$N$.
We want to prove that $E$ is in definable bijection with~$K$.

\begin{claim}\label{cl:binary}
$C$ is in definable bijection with $(0,1]$.
\end{claim}

The Corollary then follows: in fact, by Lemma~\ref{lem:defbijfinite}, $E \setminus C$ is in definable bijection
with~$N$.
Moreover, the disjoint union of $K$ and $N$ is in definable bijection with~$K$:
we define $\mu: K  \sqcup N \to K$  as follows:
\[\mu(x) :=\begin{cases}
 x   &       \text{if } x \in K \setminus N\\
 2x   &     \text{if  $x$ is in the copy of $N$ inside $K$}\\
 2x + 1 & \text{if $x$ is in the copy of $N$ outside $K$.}
\end{cases}
\]
Hence, $E$ is in definable bijection with~$K$.

Let us prove now Claim~\ref{cl:binary}.
By Lemma \ref{lem:recursive} there is a unique definable function from $N$
to $N$,
which we denote by $2^n$, such that $2^0 = 1$ and
$2^{n+1} = 2 \cdot 2^n$.
For $i \in \set{0,1}$ let
\[
Y_{n,i} := \set{a \in (0, 1] \ : \ \hbox{ there is } m \in 2 \cdot N + i
  \hbox{ such that } m <  2^n a \leq m+1}.
\]
Given $X \subset N$ definable and unbounded and $n \in N$, let
$f: N \to \set{0,1}$ be the characteristic function of~$X$, and
\[
Z_{n,X} := \bigcap_{\ell \in N_{\leq n}} Y_{\ell, f(\ell)}
\]
By induction on $n$, it is easy to see that
\[
Z_{n,X} = \Bigl(\frac{a_n}{2^n}, \frac{a_n+1}{2^n}\Bigr]
\]
for some unique $a_n \in N$ with $0 \leq a_n < 2^n$.
Let $a := \limsup_{n \to \infty} \frac{a_n}{2^n}$.
Then, since we assumed that $X$ is unbounded, it is easy to see that
$\bigcap_{n \in N} Z_{n, X} = \set a$ and $a \in (0,1]$.
Define $\hat\lambda(X) := a$.

We now show that
$\hat\lambda$ is a definable bijection between $C$ and $(0,1]$.
Given $a \in (0,1]$ let $X := \set{n \in N: a \in Y_{n,1}}$: then,
  $\hat\lambda(X) = a$, and hence $\hat\lambda$ is surjective.
Let $X, X'$ be distinct definable unbounded subsets of~$N$, and assume, for a
contradiction, that $b := \hat \lambda(X) = \hat \lambda(X')$.
Let $n := \min(X \Delta X')$; w.l.o.g., we can assume $n \in X \setminus X'$.
Then, $Z_{n,X'} = \Bigl(\frac{a_n - 1}{2^n}, \frac{a_n}{2^n}\Bigr]$ and
$Z_{n,X} = \Bigl(\frac{a_n}{2^n}, \frac{a_n + 1}{2^n}\Bigr]$ for a unique
$a_n \in N$ with $1 \leq a_n < 2^n$.
Moreover, since $b = \hat \lambda(X) = \hat \lambda(X')$, we have
$a_n = b \cdot 2^n$, and, for every $m > n$, $m \in X'$ and $m \notin X$;
however, the latter contradicts the fact that $X$ is unbounded.

The corresponding function $\lambda: K \times N \to \set{0,1}$ is given
$\lambda(a,n) = 1 \leftrightarrow a \in Y_{n,1}$.
\end{proof}


\begin{lem}\label{lem:delta}
The family of definable functions from $N$ to $K$ is in definable bijection
with~$K$.
\end{lem}
\begin{proof}
The idea of the proof is that $K^N \approx (2^N)^N \approx 2^{N \times N} \approx 2^N \approx K$,
where $A^B$ denotes the family of definable functions from $B$ to~$A$, and
$A \approx B$ means that there is a definable bijection between $A$ and~$B$.

More in details, fix a definable bijection $\theta: N \times N \to N$.
Given $f: K \to N$ definable, let $X_f$ be the definable subset of $N$ such
that, for every $i, j \in N$,
\[
j \in \hat\epsilon^{\, -1}(f(i)) \leftrightarrow \theta(i,j)  \in X_f.
\]
The definable bijection $\hat\delta$ is given by $\hat\delta(f) := \hat\epsilon(X_f)$.
Equivalently, define $\hat\delta(f)$ to be the unique $b \in K$ such that, for
every $i, j \in N$,
$\epsilon(f(i), j) = \epsilon(b, \theta(i,j))$.

The corresponding function $\delta$ is defined in the following way: for every $b
\in K$ and $i \in N$, $\delta(b,i)$ is the unique
$c \in K$ such that, for every $j \in N$, $\epsilon(c,j) = \epsilon(b, \theta(i,j))$.
\end{proof}

\begin{cor}\label{cor:recursive}
Let $c : K^n \to K$ and $g : K^{n} \times N \to K$ be definable. Then there is
a unique definable function $f : K^n \times N \to K$ such that for all $a \in K^n$
\begin{align*}
f(a,0) &= c(a),\\
f(a,i+1) &= g(a,f(a,i)).
\end{align*}
\end{cor}




Notice that from the proofs of Corollary~\ref{cor:tuples} and
Lemma~\ref{lem:delta} it follows that every definable subset of $N$ and every
definable function from $N$ to $K$ are already definable in $\< K, N, +,
\cdot, < \>$.

Moreover, 
we can encode definable subsets of $N$ as
elements of $\K$.
Thus, the set sort of the proposed model of arithmetic is $\K$ itself, and the
inclusion relation $\in$ is defined as follows:
\[
n \in a \leftrightarrow \epsilon(a, n)=1.
\]
Finally, we add a function, predicate or constant for (the translation via
$\hat\epsilon$ of) every function, predicate, or constant in the language of~$\K$.
It is now clear that the resulting structure is a model of second-order
arithmetic.

\bigskip
\subsection{From models of arithmetic to unrestrained structures}
Conversely, start with
$\mathcal N := \structure{N, D; 0, 1, +,\cdot, < ,\dotsc}$ a model of
second-order arithmetic, in the sense explained at the beginning of
\S\ref{sec:Peano}, in the language~$L$.
Let $Z$ be the ring generated by $N$ and $Q$ be the field of fractions of~$Z$.
As in \cite[Def. I.4.2]{Simpson}, a set of ``real numbers'' can be interpreted
inside $\mathcal N$: more precisely, a ``sequence of rational numbers'' is a
definable function from $N$ to~$Q$; such a sequence is Cauchy if it satisfies
the usual Cauchy condition.
We set $K$ to be the set of Cauchy sequences modulo the null sequences,
with the operations $+$, $\cdot$ and order $<$ induced by the ones on~$Q$.
Clearly, $Q$~embeds definably and canonically in~$K$.
It is also clear that $\structure{K, \cdot, + , <}$ is an ordered field;
moreover, since the family of Cauchy sequences of rational numbers
is a definable family, $K$ is interpretable in~$\mathcal N$.

\begin{lem}\label{lem:K0}
$\K_0 := \structure{K, Z, \cdot, +, <}$ is a definably complete structure.
\end{lem}
\begin{proof}
Standard proof of analysis (cf\mbox{.} \cite[Theorem III.2.2]{Simpson}).
Let $A \subset K$ be definable, bounded, and nonempty; we have to show that
$A$~has a least upper bound.
W.l.o.g., we can assume that $A$ is an initial segment, that is, if $a \in A$ and
$b < a$, then $b \in A$; moreover, we can also assume $0 \in A$.

For every $n \in N$, let
\[
f(n) := \max\set{\frac{m}{ 2^n}: m \in N, \frac{m}{2^n} \in A}.
\]
By definition, $f$~takes values in $Q \cap A$;
it is clear that $f$ is a Cauchy sequence, and that its equivalence class is
the l.u.b\mbox{.} of~$A$.
\end{proof}

Thus, we have the function $\hat\epsilon$ for the structure $\K_0$; using the
coding given by $\hat\epsilon$ we can translate all the extra functions, predicates and constants in~$L$
as functions, predicates, and constants on~$K$;
we denote by $\K$ the resulting expansion of~$\K_0$.
It is still true (with the same proof as in Lemma~\ref{lem:K0})
that $\K$ is definably complete, and thus we showed how to
transform a model of second-order arithmetic into an unrestrained definably complete structure.

The two transformations are inverse to each other; thus we can say that models
of second-order arithmetic and unrestrained definably complete structures are
the same objects.

\section{Definable functions and meager sets}
In this section we will establish some preliminary facts about definable
functions, and show how to transfer part of the theory about Baire category to
the definable context.
We will later use these facts to prove Theorem~\ref{thm:monotone}.

\subsection{Definably meager and $D_\Sigma$ sets}
In order to shows how to use Theorem~A, and because
we will use it in the remainder of this section,
we give a quick new proof of a conjecture
by Fornasiero and Servi \cite{FS}. It was first proved by different methods in~\cite{baire}.

\begin{defn}
A definable set $A \subseteq K^n$ is called \textbf{definably meager}
if $A = \bigcup_{t \in K} X_t$, for some
definable increasing family  $(X_t: t \in K)$ of nowhere dense subsets of~$K^n$.
\end{defn}

\begin{lem}\label{lem:enum-meager}
Let $A \subseteq K^n$ be \pn. Then $A$ is definably meager.
\end{lem}
\begin{proof}
Since $A$ is \pn, there exists a definable closed and discrete set $D \subset K_{\geq 0}$
and a definable surjective function $f: D \to A$. For each $t \in K$, let $X_t := f(D_{\leq t})$. By Fact \ref{imagepf}, each $X_t$ is pseudo-finite.
Then $A = \bigcup_{t \in K} X_t$, and $(X_t: t \in K)$
is definable increasing family of nowhere dense subsets of~$K^n$.
\end{proof}

\begin{thm}[\cite{baire}]\label{thm:baire}
$K$ is not definably meager.
\end{thm}
\begin{proof}
If $\K$ is restrained, then Theorem~A and
\cite[Proposition~6.4]{enum} shows that every
definably meager sets has empty interior, and in particular $K$ is not
definably meager.

If not, then, as shown in \cite[Lemma~6.2]{enum},
we can mimic one of the classical proofs of Baire's category theorem
to conclude that $K$ is not definably meager.
\end{proof}

\begin{defn}
Let $X \subseteq K^n$ be a definable set.
We say that $X$ is a \textbf{$D_\Sigma$ set} if it is the union of a definable
increasing family, indexed by~$K$, of closed subsets of $K^n$.
\end{defn}

By \cite[Remark~3.3]{FS}, a definable set is a $D_\Sigma$ set iff it is the projection of a definable
closed set.



\begin{lem}\label{fact:KU}
Let $A \subseteq K^{n+m}$ be a $D_\Sigma$ set.
Let
\[
T(A) := \set{x \in K^n: A_x \text{ is definably meager}}.
\]
Then $A$~is definably meager iff $K^n \setminus T(A)$ is definably meager.
\end{lem}
\begin{proof}
It follows immediately from \cite[Lemma 5.2 and Proposition~5.4]{FS}.
\end{proof}


\subsection{Definable functions and continuity}
Now that we have a reasonable analogue of the notion of meager sets, we can
use it to transfer several well-known results from real analysis to~$\K$.
For the remainder of this subsection we will not use Theorem~A anymore.
Afterwards, we will use these results to prove
Theorem~B by distinguishing the case when $\K$ is either restrained or
unrestrained.

First, we show that a monotone function $f$ is continuous outside a ``small'' set.

\begin{lem}\label{lem:monot-cont}
Let $f: K \to K$ be a definable monotone function.
Then, the set $\Disc f$ of discontinuity points of $f$ is \pn.
\end{lem}
\begin{proof}
For every $\eps > 0$ let
\begin{align*}
\Disc f(\eps) &:= \set{x \in K: \limsup_{y \to x}\abs{f(y) - f(x)} > \eps}\\
&=
\set{x \in K: \lim_{y \to x^+} f(y) - \lim_{y \to x^-} f(y) > \eps}.
\end{align*}
It is easy to see that $f(\Disc f(\eps))$ is discrete for every $\eps > 0$.
Thus, by Fact~\ref{lem:discrete-union}, since 
$ \Disc f = \bigcup_{\eps > 0} \Disc f(\eps)$, $\Disc f$ is \pn.
\end{proof}

\begin{defn}
Let $f: K \to K$ be a definable function.
The four \textbf{Dini derivatives} of $f$ are:
\[
\begin{aligned}
\llf(x) & := \liminf_{y \to x^-} \frac{f(y) - f(x)}{y - x}\\
\lrf(x) & := \liminf_{y \to x^+} \frac{f(y) - f(x)}{y - x}\\
\Llf(x) & := \limsup_{y \to x^-} \frac{f(y) - f(x)}{y - x}\\
\Lrf(x) & := \limsup_{y \to x^+} \frac{f(y) - f(x)}{y - x}.
\end{aligned}\]
\end{defn}

\begin{lem}\label{lem:Dini}
Let $f: K \to K$ be definable and continuous.
If, for every $x \in K$, $ \Lrf(x) \in K$ and $\Lrf$ is continuous, then $f$ is $\Cone$ (and $f' = \Lrf$).
\end{lem}
\begin{proof}
As in \cite[Theorem~1.3]{Bruckner}.
\end{proof}

We will now adapt the classical definition of Baire class to the ``definable'' context.

\begin{defn}
Let $X \subseteq K^n$ be a definable set, $f: X \to K$ be a definable
function, and $n \in \N$.
We say that $f$ is of \textbf{definable Baire class $n$} if:
\begin{enumerate}
\item either $n = 0$ and $f$ is continuous;
\item or $n > 0$ and there exists a definable family of functions
$(f_t: X \to K)_{t \in K}$ such that each $f_t$ is of class $(n - 1)$ and
\begin{enumerate}
\item either,
for every $x \in X$, $f(x) = \lim_{t \to + \infty} f_t(x)$.
\item or, for every $x \in X$, $f(x) = \sup_t f_t(x)$;
\item or, for every $x \in X$, $f(x) = \inf_t f_t(x)$.
\end{enumerate}
\end{enumerate}
\end{defn}
In the above definition we had to add Clauses (2-b) and (2-c) to
the classical definition,
because we could not prove that a function satisfying e.g.\ Clause (2-b) would
satisfy Clause (2-a).



The interest for us of the above definition stems from the following fact.

\begin{lem}\label{lem:Dini-Baire-2}
Let $f: K \to K$ be definable and continuous.
Then, $\Lrf$ is of definable Baire class~$2$.
\end{lem}
\begin{proof}
For every $t \neq 0$ let $g_t(x) := \frac{f(x + t) - f(x)}{t}$.
Then,
\[
\Lrf(x) = \inf_{t > 0}\, \sup_{0 < s < t} g_s(x). \qedhere
\]
\end{proof}

\begin{defn}
Let $f: X \to K$ be a definable function.
We say that $f$ is \textbf{almost continuous} if the set of its discontinuity
points $\Disc f$ is nowhere dense.
\end{defn}

\begin{defn}
$\K$ has locally o-minimal open core if
there does not exist a definable, closed, discrete, and
unbounded subset of~$K_{\geq 0}$ (see \cite[Thm. 3.3]{local}).
\end{defn}

We could prove the following lemma only under
the assumption that $\K$ does \emph{not} have locally o-minimal open core.

\begin{lem}\label{lem:Baire-1}
Assume that $\K$ does not have locally o-minimal open core.
Let $(f_t: K^n \to [0,1])_{t \in K}$ be a definable family of
almost continuous functions.
Let $f: K^n \to [0,1]$ be either of the following functions:
\begin{enumerate}
\item $f(x) = \sup_t f_t(x)$;
\item or $f(x) = \lim_{t \to \infty} f_t(x)$.
\end{enumerate}
Then, the restriction of $f$ to the complement of a definably meager set is continuous.\\
If moreover each $f_t$ is continuous (i.e., $f$ is of definable Baire class~1),
then $\Disc(f)$ is definably meager.
\end{lem}
\begin{proof}
Minor variation of \cite[Thm.~7.3]{Oxtoby}.
Let $M \subset K_{\geq 0}$ be definable, closed, discrete, and unbounded.

Let $D_i$ be the closure of $\Disc f_{i}$, and $D := \bigcup_{i \in M} D_i$
since each $D_i$ is nowhere dense, $D$~is definably meager.
Let $X := K^n \setminus D$: notice that $X$ is dense in~$K^n$.
We claim that $f \rest X$ is continuous
outside a definably meager set.
(If each $f_i$ is continuous, then $D$ is empty,
and we also obtain the``moreover'' clause).

For every $\eps > 0$, set
\[
F_\eps := \set{a \in X: \forall \delta > 0\ \exists x \in X\
(\abs{x -a} < \delta  \ \&\ \abs{f(x) - f(a)} > 5 \eps)}.
\]
it suffices to show that  $F_\eps$ is nowhere dense.
Fix an open box $V \subseteq K^n$ and $\eps > 0$.

We prove first Case (2).
Notice that $f(x) = \lim_{t \to \infty, t \in M} f_t(x)$.
For every $i \in K$, let
\[
E_i := 
\set{x \in V: \abs{f_i(x) - f(x)} \leq \eps}.
\]
Notice that $(E_i: i \in M)$ is a definable family of subsets
of~$V$, and $\bigcup_{i \in M} E_i = V$.
Hence, by Theorem~\ref{thm:baire}, there exists $i_0 \in M$ such that the
closure of $E_{i_0}$ has nonempty interior.
Let $U \subseteq \cl(E_{i_0})$ be a nonempty open box.
Since $f_{i_0}$ is continuous on~$U \cap X$, after shrinking $U$
we can also assume that, for every $x, x'\in U \cap X$,
$\abs{f_{i_0}(x) - f_{i_0}(x')} \leq \eps$.
Thus, for every $x, x' \in U \cap X$, $\abs{f(x) - f(x')} \leq 3 \eps$, and
therefore $U \cap X \cap F_{\eps} = \emptyset$.

Thus, every nonempty open definable set $V$ contains a nonempty open set $U$
disjoint from $F_\eps \cap X$, and therefore $F_\eps \cap X$ is nowhere dense.

The proof of Case (1) is similar, using instead
\[
E_i := 
\set{x \in V: f(x) \leq f_i(x) + \eps}.
\qedhere
\]
\end{proof}

\begin{bsp}
\begin{enumerate}
\item
Notice that in the above lemma we cannot conclude that $\Disc f$ is definably
meager without also assuming that either  each $f_i$ is continuous or $\K$
is restrained (see Corollary~\ref{cor:Baire-restrained-continuous}).
In fact, it is easy to see that  the characteristic function of an \pn
set is the pointwise limit of a definable family of functions $f_i$ such that
each $\Disc(f_i)$ is pseudo-finite.
For instance, let $\mathcal R := \structure{\R, + , \cdot, <, \N}$.
Let $f: \R \to \R$ be the characteristic function of~$\Q$: then,
$f$ can be written as the limit of a definable family of functions $f_i$, with
$\Disc(f_i)$ finite for every~$i$.
\item Let $f: K^n \to K$ be a definable continuous function such that $\Lrf$ is
discontinuous on a nonmeager set: then $\Lrf$ is a function of definable Baire
class exactly~2 (i.e., not~1).  It is an easy exercise to find such a function
$f$ when $\K$ is unrestrained (cf\mbox{.} \cite[p.~42]{Bruckner}): however, we
will see later that when $\K$ is restrained such $f$ does not exist.
\item
If $X \subseteq K^n$ is a nonempty definable closed
subset, then the characteristic function of $X$ is of definable Baire class~1.
\item
If $\mathcal R$ is an unrestrained expansion of the real field, then, for each
$n \in \N$,
the ``definable Baire class $n$'' and the ``Baire class $n$'' are the same class
(because all sets in the projective hierarchy are definable in~$\mathcal R$);
therefore, by a theorem by Lebesgue \cite{Lebesgue}, for each $n$ there is a definable
function of definable Baire class exactly~$n$.
\item
Let $C \subset \R$ be a ``Cantor set'', i.e., a nonempty, definable, closed,
perfect, nowhere dense subset.
Define $f: \R \to \R$ as $f(x) = 0$ outside~$C$, $f(x) = 1/2$ on each point of
$C$ such that there exists $\eps > 0$ with either $(x, x + \eps) \cap C =
\emptyset$, or $(x -\eps, x) \cap C = \emptyset$, and $f(x) = 1$ otherwise.
Then, $f$ is of definable Baire class exactly~$2$ (cf\mbox{.} \cite[Ch.7]{Oxtoby}).
There are some restrained expansions of $\R$ defining a Cantor set as above.
We leave open the question if in the restrained case there
can be definable functions of definable Baire class greater than~2.
\item
If $\K$ is restrained and defines set $X \subset K$ which is  dense and
codense, then the
characteristic function of $X$ is not in any definable Baire class; for
instance, if $\mathcal R$ is the expansion of the real field by the set
$R^{alg}$ of real algebraic numbers, then the characteristic function of
$R^{alg}$ is of Baire class~2, but it is not in any definable Baire class.
\end{enumerate}
\end{bsp}

\subsection{Restrained structures}
\label{subsec:restrained}
In this subsection we will prove a few results about definable functions and
sets in restrained structures.
We will use them to prove the restrained case of Theorem~\ref{thm:monotone};
however, we think that some of them are of independent interest.

\begin{lem}\label{lem:nd-countable}
Let $X \subseteq K$ be definable and nowhere dense.
Then, there exists two sets $Y, Z \subset K$ discrete, definable, and such that
 $Y \subseteq X$ and $\cl(X) \subseteq \cl(Y) \cup \cl(Z)$.
Moreover, the choice of $Y$ can be made in a uniform way:
that is, if $X \subset K^{n+1}$ is definable, and for every $t \in K^n$,
$X_t$ is nowhere dense, then there exists $Y, Z \subset K^{n+1}$ definable,
such that $Y \subseteq X$ and, for every $t \in K^n$, $Y_t$ and $Z_t$ are
discrete, and $X_t \subseteq \cl(Y_t) \cup \cl(Z_t)$.
\end{lem}
\begin{proof}
Let $Y$ be the set of isolated points of~$X$.
W.l.o.g., we can assume that $X$ is closed and $X \subset (0,1)$.
Thus, $(0,1) \setminus X$ can be written in a unique way as a union of
disjoint open intervals; let $Z$ be the set of centers of such
intervals. 
\end{proof}

\begin{lem}\label{lem:restrained-meager}
$\K$ is restrained iff, for every $m \in \N$,
every definably meager subset of $K^m$ is nowhere dense.
\end{lem}
\begin{proof}\setcounter{claim}{0}
For the ``if'' direction,
let $X \subset K$ be \pn.
Then, by Lemma~\ref{lem:enum-meager}, $X$~is definably meager;
thus, by assumption,
$X$~is nowhere dense, proving that $\K$ is restrained.

For the ``only if'' direction, first we assume $m = 1$.
If $\K$ has locally o-minimal open core, then the conclusion holds (see
\cite[Theorem 3.3]{local}).
Otherwise, there exists an unbounded definable closed
discrete set $D \subset K_{\geq 0}$.
Let $X \subset K$ be definably meager; thus, $X = \bigcup_{i \in K} Y_i$,
for some $(Y_i: i \in K)$ definable increasing family of nowhere dense set.
Since $D$ is unbounded, $X = \bigcup_{i \in D} Y_i$.
By Lemma~\ref{lem:nd-countable}, there exists two definable families of discrete
sets $(Z_i: i \in D)$ and $(W_i: i \in D)$, such that, for every $i \in D$,
$Y_i \subseteq \cl(Z_i \cup W_i)$.
Let $T := \bigcup_{i \in D} Z_i \cup W_i$.
By Fact~\ref{lem:discrete-union},
$T$~is \pn, and hence nowhere dense, since $\K$ is restrained.
Since $X \subseteq \cl(T)$, we have that $X$ is nowhere dense.

Assume now that $m \geq 1$ (and $\K$ is restrained).
By induction on $n$, we show the following:
\begin{enumerate}
\item[$(1)_n$]
Every $D_\Sigma$ subset of $K^n$ has interior or is nowhere dense;
\item[$(2)_n$]
For every $p \in \N$ and $A$ $D_\Sigma$ subset of $K^{n + p}$, the set
$\set{x \in K^n: \cl(A_x) \neq \cl(A)_x}$ is definably meager in~$K^n$.
\item[$(3)_n$] If $A$ is a $D_\Sigma$ subset of $K^n$, then $\fr(A) :=
\cl(A) \setminus \interior A $ is nowhere dense.
\item[$(4)_n$] Every definably meager subset of $K^n$ is nowhere dense.
\end{enumerate}
Assertion $(4)_m$ is the conclusion of the Lemma.
Assertion $(1)_1$ is the Case $m = 1$.

The proofs of $(2)_1$ and the inductive step are as in \cite[1.6]{MS}.

More precisely, assume that we have already proved $(1)_n$;
we claim that $(2)_n$,  $(3)_n$, and $(4)_n$ also hold.
For $(3)_n$: we have
\[
\fr(A) = \fr(\interior A ) \cup \fr(A \setminus \interior A)
= \fr(\interior A ) \cup \cl (A \setminus \interior A)
\]
and each of the two pieces is a $D_\Sigma$ set with empty interior, and thus,
by $(1)_n$, nowhere dense.

For $(4)_n$, let $X \subseteq K^n$ be definably meager: that is, $X =
\bigcup_{t \in K} Y_t$, where $(Y_t: t \in K)$ is a definable increasing
family of nowhere dense subsets of~$K^n$.
For each $t \in K$, let $Z_t$ be the closure of $Y_t$ (inside $K^n$); define
$W := \bigcup_t Z_t$.
Then, $W$ is definably meager and hence, by Theorem~\ref{thm:baire}, with empty interior; moreover, $W$~is
a $D_\Sigma$ set.
Thus, by $(1)_n$, $W$~is nowhere dense, and, since $X \subseteq W$,
$X$~is also nowhere dense.

The proof of $(2)_n$ is a bit more involved.
Let $A$ be a $D_\Sigma$ subset of $K^{n+p}$ and
$B := \set{x \in K^n: \exists y \in \cl(A)_x \setminus \cl(A_x)}$.
We want to show that $B$ is definably meager.

For each open box $U \subseteq K^p$, let
$C_U := \set{\pair{x,y} \in \cl(A): y \in U \ \& \ \cl(A_x) \cap U = \emptyset}$
and $B_U := \pi(C_U)$, where $\pi: K^{n+p} \to K^n$ is the projection onto
the first $n$ coordinates.
Notice that $B$ is the union of all the $B_U$'s.
\begin{claim}\label{cl:BU}
For each open box $U$, $B_U$ is nowhere dense.
\end{claim}
In fact, let $G := \pi(A \cap (K^n \times U))$.
Then, $G$ is a $D_\Sigma$ set, and $\fr(G)$ has empty
interior (by $(3)_n$).
However, $B_U \subseteq \fr(G)$, and the claim is proved.

For each $r >0$, let
\[
D(r) := \set{\pair{x,y} \in \cl(A): \abs y \leq r \ \& \ d(y, A_x) \geq r},
\]
$E(r) := \cl(D(r))$, and $F(r) := \pi(E(r))$.
Since $B = \bigcup_{r > 0} \pi(D(r)) \subseteq \bigcup_{r>0} F(r)$, and each
$F(r)$ is closed,
it suffices to show that each $F(r)$ has empty interior.
Assume, for a contradiction, that $F(r)$ contains a nonempty
open box~$V$, for some $r > 0$.
Define $f: V \to K^p$, $f(x) := \lexmin (E(r)_x)$.
By \cite[2.8(1)]{DMS10}, the set of discontinuity points of $f$ is definably
meager; thus, by $(1)_n$, after shrinking $V$ if necessary, we can assume that
$f$ is continuous on~$V$.
Thus, $\Gamma(f)$, the graph of~$f$, is contained in $E(r)$.
After shrinking $V$ if necessary, by continuity of~$f$,
we can find an open box $U \subset K^p$
of diameter less than~$r$ and such that $f(V) \subseteq U$.

Then, $D(r)_U := D(r) \cap (K^n \times U) \subseteq C_U$, and therefore
\[
V \subseteq \pi\Pa{\cl(D(r)_U)}
\subseteq \cl\Pa{\pi(D(r)_U)}
\subseteq \cl(B_U),
\]
contradicting Claim~\ref{cl:BU}.

Finally, assume that we have already proved all the statements for every $n' < n$; we
want to prove $(1)_n$.
Let $A \subset K^n$ be a $D_\Sigma$ set with empty interior;
we want to prove that $A$ is nowhere dense.
Notice that $A$ is definably meager; thus,
by Lemma~\ref{fact:KU}, the set of points $x \in K^{n-1}$ such that $A_x$ has
nonempty interior is definably meager; hence, by $(1)_1$ and $(4)_{n-1}$,
the set of points $x \in K^{n-1}$ such that $A_x$ is somewhere dense
is nowhere dense.
By $(2)_{n-1}$, the set of points $x \in K^{n-1}$ such that
$\cl(A)_x$ has interior is nowhere dense.
Hence, $\cl(A)$ has empty interior.
\end{proof}

\begin{cor}\label{cor:Baire-restrained-continuous}
Let $\K$ be restrained and without locally o-minimal open core, $n, m \in \N$,
and $f: K^m \to K$ be of definable Baire class~$n$.
Then, $f$ is almost continuous.
\end{cor}
\begin{proof}
By induction on $n$, Lemma~\ref{lem:Baire-1}, and
Lemma~\ref{lem:restrained-meager}.
\end{proof}

\begin{lem}\label{lem:restrained-Cone}
Let $\K$ be restrained, $U \subseteq K^n$ be open and definable,
$f: U \to K$ be a definable continuous function,
and $p \in \N$.
Then, $f$ is $\mathcal C^p$ on a dense open subset of~$U$.
\end{lem}
\begin{proof}\setcounter{claim}{0}
Let $B \subseteq U$ be an open box; it suffices to prove the result for $f
\rest B$; since $B$ is diffeomorphic to $K^n$ itself, it suffices to treat the
case when $U = K^n$.

If $\K$ has locally o-minimal open core, then, since $f$ is definable in the
open core of~$\K$, the conclusion follows from \cite[Theorem~5.11]{local}.

Otherwise, by induction, it suffices to treat the case $p = 1$.
First, we do the case $n = 1$.
By Lemma~\ref{lem:Dini-Baire-2}, $\Lrf$ is of definably Baire class $2$.
By Corollary~\ref{cor:Baire-restrained-continuous},
$\Lrf: K \to K \cup \set{\pm \infty}$ is continuous on a dense open
set~$U$, but may take value infinity somewhere.
\begin{claim}
Let $V := \set{x \in U: \Lrf(x) \in K}$.
Then, $V$~is open and dense.
\end{claim}
If not, since $\Lrf$ is continuous on $U$,
there would exist an interval $[a, b] \subseteq U$ such that
\begin{enumerate}
\item either for every $x \in [a, b]$, $\Lrf(x) = + \infty$,
\item or, for every $x
\in [a, b]$, $\Lrf(x) = - \infty$.
\end{enumerate}
By replacing $f(x)$ with $f(x) - \frac{f(b) - f(a)}{b - a} (x - a)$, w.l.o.g.\ we
can assume that $f(b) = f(a)$.
Thus, since $f$ is continuous and definable, there exists $x_0 \in (a,b)$ that
is a maximum for $f$ in $[a,b]$; but then $\Lrf(x_0) \leq 0$, contradicting
Case~(1).
Similarly, there exists $x_1 \in (a,b)$ that is a minimum for $f$ in $(a,b)$,
contradicting Case~(2).
Finally, by Lemma~\ref{lem:Dini}, $f$ is $\Cone$ on~$V$.

Assume now that $n > 1$.
We will prove that, outside some nowhere dense set, each partial derivative of
$f$ exists and is continuous; it suffices to show that $\partial f / \partial
x_n$ exists and is continuous on a dense open set.
Let $\bar e_n := \< 0, \dotsc, 0, 1 \> \in K^n$.
Define the Dini derivatives of $f$ in the direction $\bar e_n$ as $\Lrf :=
\limsup_{t \to   0^+} \frac{f(x + t \bar e_n) - f(x)} t$, and similarly for
the other three Dini derivatives.
Reasoning as in the case $n = 1$,  we see that $\Lrf$ is finite and continuous
on a dense open set~$U$, and similarly for the other three Dini derivatives.
It then suffices to show that, after maybe shrinking $U$
to a smaller dense open definable set, the four Dini derivatives coincide; by
symmetry, it suffices to prove that $\llf  = \Lrf$ on a dense open set.
Assume not: then, by continuity, there would exists an open set $V$ such that
$\llf(x) \neq \Lrf(x)$ for every $x \in V$; but this contradicts the case $n = 1$.
\end{proof}

\section{Lebesgue's Theorem}
\label{sec:thmB}
We give now an application of Theorem A, by proving the following
analogue of Lebesgue's theorem.
Remember that we call $\K$ unrestrained if it defines a discrete subring
(with~$1$), and restrained otherwise.

\begin{thmB}
Let $f: K \to K$ be a definable monotone function.
Then, $f'(x)$ exists and is in $K$ (i.e., not $\pm \infty$)
on a dense subset of~$K$.
\end{thmB}

The reasons we chose this example are that it is interesting in its own
right (it was conjectured in \cite{ivp}), and it gives a good
illustration of how Theorem A can be used to transfer
various classical results from $\R$ to~$\K$. Theorem A allows us to reduce
the proof of the above Theorem to structures
satisfying either condition (I) or (II) of Theorem A.

\subsection{The restrained case}

We will now give a proof of Theorem~\ref{thm:monotone} in the case when $\K$ is restrained.

The theorem in the restrained case follows immediately from the results in
\S\ref{subsec:restrained} plus the following lemma.
\begin{lem}\label{lem:restrained-continuous}
Let $\K$ be restrained; let $f: K \to K$ be a definable monotone function.
Then, there exists a definable closed nowhere dense set $C$ such that $f$ is
continuous outside~$C$.
\end{lem}
\begin{proof}
Let $D$ be the set of discontinuity points of $f$, and $C$ be its closure.
By Lemma~\ref{lem:monot-cont}, $D$ is \pn;
by Theorem~A, $C$ is nowhere dense.
\end{proof}


The following corollary concludes the proof of Theorem~B in the case when
$\K$ is restrained.
\begin{cor}
Let $\K$ be restrained; let $f: K \to K$ be a definable monotone function.
Then, $f$ is $\Cone$ outside a nowhere dense set.
\end{cor}
\begin{proof}
By Lemmas~\ref{lem:restrained-continuous} and~\ref{lem:restrained-Cone}.
\end{proof}

\subsection{Measure theory}
\label{sub:measure}
Let us examine now the case when $\K$ defines a discrete subring~$Z$.
Using the results in \S\ref{sec:Peano}, we can transfer the tools of measure theory.
We will sketch the relevant ideas in the following (cf.\ \cite[\S
X.1]{Simpson} for a different approach).
Many of the definitions make sense also in the case when $\K$ is restrained:
therefore in this subsection, unless said
otherwise, we are not assuming that $\K$ is unrestrained.

\begin{defn} Let $D \subset K_{\geq 0}$ be a nonempty closed discrete definable set, and
let $s_D$ be defined as in Definition~\ref{def:successor}.
Let $h: D \to K$ be a definable function. We define $H: D \to K$ to be function given recursively by
$H(\min(D)) = 0$ and for every $d \in D$ with $d \neq \max(D)$,
$H(s_D(d)) = H(d) + h(d)$.
If $h$ takes only nonnegative values and $H$ exists, we denote
\[
\sum_{d \in D} h(d) := \sup_{d \in D} H(d) \in K_{\geq 0} \cup\set{+ \infty}.
\]
\end{defn}
It is easy to see that if $H$ is definable, then it is unique. Moreover, if
 $\K$ is unrestrained, then $H$ exists by Corollary \ref{cor:recursive}.

\begin{defn}[Lebesgue measure]
Let $a < b \in K\cup \set{\pm \infty}$; we set
$\abs{(a, b)} := b - a$.
Let $\mathcal U := \Pa{I_d: d \in D}$ be a definable family of open intervals,
indexed by a closed discrete set $D \subseteq K_\geq 0$.
We define $M(\mathcal U) := \sum_{d \in D} \abs{I_d}$ (if it exists).
\end{defn}

Let $A \subseteq K$ be a definable set.
We denote by $\mu(A)$ the infimum of
$M(\mathcal U)$, as $\mathcal U$ varies among all the definable coverings of
$A$ by open intervals, indexed by some definable discrete subset of $K_{\geq
  0}$, such that $M(\mathcal U)$ exists.
Notice that $\mu(A)$ may not lie in $K$ (since it is the infimum of a set
that may not be definable), but in the Dedekind-MacNeille completion
of~$\structure{K, <}$.
Notice also that $0 \leq \mu((0,1)) \leq 1$.

Notice that when $\K$ expands $\<\R, +, \cdot, \N \>$, then $\mu(X)$ is the outer Lebesgue measure of~$X$.

\begin{conjecture}
$\mu((0,1)) = 1$.
\end{conjecture}

However, things are much simpler if $\K$ unrestrained.
In that case, $M(\mathcal U)$ always exists, and we can always assume that
the index set of $\mathcal U$ is either $N$ or an initial segment of~$N$
(more precisely, for every definable closed discrete subset
$D \subset K_{\geq 0}$ there is a unique definable increasing bijection
between a unique initial segment of $N$ and~$D$).

Moreover, the family of definable covers of a given definable set $A$ by open
intervals indexed by $N$ is itself definable (by Lemma~\ref{lem:delta}),
and therefore $\mu(A) \in K_{\geq 0} \cup \set{+ \infty}$.
Moreover, again by using Lemma~\ref{lem:delta}, if $(A_i: i \in I)$ is a
definable family, then $f: i \mapsto \mu(A_i)$ is a definable function.

\begin{prop}
Let $\K$ be unrestrained.
Then, $\mu((0,1)) = 1$.
\end{prop}

The proof of the above proposition is a minor modification the classical one
that $(0,1)$ has Lebesgue measure~$1$, and is left to the reader;
he can base it on the following result, whose proof is also left to the reader.

\begin{lem}[Commutativity of addition]
Let $\K$ be unrestrained.
Let $h : N \to K_{\geq 0}$ be a definable function, and
$\sigma: N \to N$ be a definable bijection.
Then, $\sum_{d \in N} h(d) = \sum_{d \in N} h(\sigma(d))$.
\end{lem}

Notice that we are not able to prove the above lemma without
the assumption that $\K$ defines a discrete subring.

\begin{conjecture}
Let $D \subseteq K_{\geq 0}$ be a definable closed discrete subset;
let $h: D \to K_{\geq 0}$ be a definable function, and
$\sigma: D \to D$ be a definable bijection.
Then, $\sum_{d  \in D} h(d) = \sum_{d \in D} h(\sigma(d))$
(i.e., if the sum on the left exists, then also the one on the right exists
and is equal to it).
\end{conjecture}

\begin{lem}[Sigma-subadditivity of measure]
Let $\K$ be unrestrained.
Let $\Pa{A_i: i \in N}$ be a definable family of subsets of~$K$.
Then,
\[
\mu(\bigcup_i A_i) \leq \sum_i \mu(A_i).
\]
In particular, if $\mu(A_i) = 0$ for every $i \in N$, then
$\mu(\bigcup_i A_i) = 0$.
Therefore, if $A \subset K$ is \pn, then $\mu(A) = 0$.
\end{lem}
\begin{proof}
Left to the reader.
\end{proof}

\begin{cor}\label{cor:local-null}
Let $\K$ be unrestrained.
Let $X \subseteq K$ be a definable set, and $0 \leq \delta < 1 \in K$.
Assume that for every interval $I$ we have
$\mu(X \cap I) \leq \delta \abs I$.
Then, $\mu(X) = 0$.
\end{cor}
\begin{proof}
Assume not: let $\mu(X) = c > 0$.
Fix $0 < \eps \in K$ small enough (how small will be clear later).
Let $\mathcal U := (I_d : d \in N)$ be a definable family of intervals, such
that $M(\mathcal U) < (1 + \eps) c$ and $X \subseteq \bigcup_d I_d$.
Thus, by our assumption on~$X$,
\[
\mu(X) \leq \sum_d \mu(I_d \cap X) \leq \sum_d \delta \abs {I_d} \leq \delta
(1 + \eps) c.
\]
If we take $\eps$ small enough, we have $ \delta (1 + \eps) < 1$, absurd.
\end{proof}

\subsection{The unrestrained case}

With those tools at our disposal, we can now mimic some of the proofs of
Lebesgue's theorem: we will follow the trace of \cite{Riesz} for the case when
$f$ is continuous, and of \cite{Rubel} for the general cases.

First, a technical lemma, which is easy to prove for every~$\K$, without using
Theorem~A: the proof is left to the reader (cf.~\cite{Riesz} for the details).
\begin{lem}[Riesz's Rising Sun Lemma]\label{lem:sun}
Let $a < b \in K$ and $g: [a, b] \to K$ be a definable bounded function.
For every $x \in [a, b]$, denote
\[
G(x) := \max \Pa{g(x), \limsup_{y \to x} g(y)}.
\]
Let
\[
E := \set{x \in (a, b): (\exists y \in (x, b])\ g(y) > G(x)}.
\]
Then, $E$ is an open definable subset of $(a, b)$.
Moreover, let $(a', b')$ be a maximal open subinterval of~$E$.
Then, $\limsup_{y \to {a'}^+} g(y) \leq G(b')$.
\end{lem}

\begin{lem}\label{lem:monot-bounded-der}
Let $\K$ be unrestrained.
Let $a < b \in K$, and
$f: (a, b) \to K$ be a definable increasing function.
Define
\[
A_\infty := \set{x \in (a,b): \Lrf(x) = + \infty }.
\]
Then, $\mu(A_\infty) = 0$.
\end{lem}
\begin{proof}
The same as in \cite[Assertion 1]{Riesz}.

More in details, given $c \in K$, define
\[\begin{aligned}
g(x) &:= f(x) - cx;\\
A_c &:= \set{x \in (a,b): \Lrf(x) > c};\\
%
E_c & := \set{x \in (a, b): (\exists y > x)\ g(y) > g(x^+)};\\
\Disc f &:= \set{x \in (a, b): f \text{ is discontinuous at } x}.
\end{aligned}\]
Notice that $\bigcap_c A_c = A_\infty$ and $\mu(\Disc f) = 0$
(because $\Disc f$ is \pn),
and therefore it suffices to show that
$\mu(A_c \setminus \Disc f)$ is arbitrarily small for $c$ large enough.
Moreover, $A_c \setminus \Disc f \subseteq E_c$;
therefore, it suffices to show that $\mu(E_c)$ is small.

Let $G$ be as in Lemma~\ref{lem:sun}; notice that
$G(x) = g(x^+)$, unless $x = b$, when $G(b) = g(b)$.
Thus, by Lemma~\ref{lem:sun}, $E_c$ is an open subset of $(a, b)$, and it
is the disjoint union of a definable family of open intervals
$\set{(a_k, b_k): k \in N} $, such that
$c(b_k - a_k) \leq f(b_k^+) - f(a_k^+)$.
Hence, $c \sum_{k \in N} (b_k - a_k) \leq f(b) - f(a)$,
and therefore $\mu(E_c) \leq \frac{f(b) - f(a)}{c}$.
\end{proof}

\begin{lem}\label{lem:monot-cont-diff}
Let $\K$ be unrestrained.
Let $f: K \to K$ be a definable monotone continuous function.
\begin{enumerate}
\item 
Let $A := \set{x \in K: \llf(x) < \Lrf(x)}$.
Then, $\mu(A) = 0$.
\item
The set of points $x \in (a,b)$ such that $f'(x)$ does not exist
or is infinite has measure $0$.
\end{enumerate}
\end{lem}
\begin{proof}
We proceed as in \cite[Assertion~2]{Riesz}.
(2) follows easily from (1), thus we only need to prove (1).

It suffices to show that, for every $0 < c < C \in K$,
the set
\[
B := \set{x \in K: \llf(x) < c \ \&\ \Lrf(x) > C}
\]
has measure~$0$.
Let $\delta := c/C$: by Corollary~\ref{cor:local-null}, it suffices to show
that, for every $a < b \in K$, $\mu(B \cap (a, b)) < \delta (b - a)$.
As in\cite[Assertion~2]{Riesz}, by applying Lemma~\ref{lem:sun} to the
function $g(x) := f(-x) + c x$, we get that
$\set{x \in (a,b): \llf(x) <  c}$ is contained in an open definable
set~$D$, such that for every maximal interval $(a_k, b_k) \subseteq D$,
we have $f({b_k}) - f({a_k}) \leq c (b_k - a_k)$ (notice that we can take
the indexes $k$ in $N$ in a definable way).
We then apply again Lemma~\ref{lem:sun} to the function
$g(x) := f(x) - C x$ restricted to each interval $(b_k, a_k)$, and we get
that $D \cap (b_k, a_k)$ is contained in an open definable set $D_k$,
such that $\mu(D_k) \leq \frac{f(b_k) - f(a_k)} C$.
Thus,
\[
\mu(B \cap (a, b)) \leq \sum_k \mu(D_k) \leq \frac{\sum_k f(b_k) - f(a_k)} C
\leq  \delta \sum_k (b_k - a_k) \leq \delta (b - a).
\qedhere
\]
\end{proof}


Let us treat now the case when $f$ is not continuous: we will follow the ideas
in~\cite{Rubel}.

\begin{lem}\label{lem:monot-inverse}
Let $f: [a, b] \to K$ be a  strictly increasing definable function.
Then, $f(x)$ has a continuous definable inverse; that is, there exists a
continuous, non-decreasing, definable function $F$ defined on $[f(a), f(b)]$,
such that $F(f(x)) = x$ for every $x \in [a, b]$.
\end{lem}
\begin{proof}
Define $F(y) := \sup\set{t: f(t) \leq y}$.
\end{proof}


\begin{lem}
Let $\K$ be unrestrained and $a < b \in K$.
Let $f: [a,b] \to K$ be a nondecreasing definable function.
Let $E$ be the set of $x \in [a,b]$ such that either $f'$ does not exists or it is
infinite.
Then, $\mu(E) = 0$.
\end{lem}
\begin{proof}
By replacing $f(x)$ with $f(x) + x$, w.l.o.g.\ we can assume that $f$ is
strictly increasing.
Thus, we can apply Lemma~\ref{lem:monot-inverse}: let $F$ be defined there.
By lemmas~\ref{lem:monot-cont-diff} and~\ref{lem:monot-bounded-der}, $F'$ exists
and is finite outside a definable set of measure~$0$.
Given $x \neq y \in [a,b]$, we write
\[
\frac{f(y) - f(x)}{y - x} = \left(
\frac{F(f(y)) - F(f(x))}
{f(y) - f(x)}
\right)^{-1}.
\]
Thus, if $f$ is continuous at $x$ and $F'(x)$ exists, we have that
$f'(x) = 1/F(f'(x)) \in K \cup \set{+ \infty}$.
However, by Lemma~\ref{lem:monot-cont}, the set of discontinuity points of $f$
is \pn, and \emph{a fortiori} of measure $0$, and by
Lemma~\ref{lem:monot-bounded-der}, $f'(x) < + \infty$ outside a set of
measure~$0$.
\end{proof}

\begin{cor}
Let $\K$ be unrestrained.
Let $f: K \to K$ be a definable monotone function.
Let $E$ be the set of $x \in K$ such that $f'(x)$ does not exists or is
infinite.
Then, $\mu(E) = 0$, and therefore $E$ has empty interior.
\end{cor}

\subsection{Other problems}

Lest the reader thinks the transfer from the real case to the definably
complete one is always automatic, we will conclude with an open problem
and recall some counterexamples. 

\begin{conjecture}[Brouwer's Fixed Point]
Let $f: [0,1]^2 \to [0,1]^2$ be a definable continuous function.
Then, $f$ has a fixed point,
i.e.\ there exists $c \in [0,1]^2$ such that $f(c) = c$.
\end{conjecture}

\begin{fact}[Hrushovski, Peterzil \cite{Hrushovski-Peterzil}]
There exists an o-minimal structure $\K$
and a definable $\mathcal C^\infty$
nonzero function $f: I \to K$, where $I$ is an open interval around $0$,
such that $f(0) = 0$ and 
$f$ satisfies the differential equation
\begin{equation}\label{eq:hp}
f(x) = x^2 f'(x) + x, \quad
f(0) = 0, \quad  f'(0) = 1.%
\footnote{$f$ is the formal power series $\sum_{n \geq 1} (n-1)!\, x^{n}$.}
\end{equation}
For every $0 < \eps \in \R$ there is no $\Cone$ function
$f: (-\eps, \eps) \to \R$ satisfying Equation~\eqref{eq:hp}.
\end{fact}

Another counterexample to some kind of ``transfer principle'' for restrained
(indeed, locally o-minimal) structures can be found in work by Rennet in \cite{Rennet}.

For unrestrained structures, let $\mathcal R$ be an expansion of
$\<\R, +, \cdot, <, \N \>$, $L$~be its language,
$T_0$~be the $L$-theory whose models are
definably complete structures with a discrete subring, and $T$ be any
recursive set of sentences true in $\mathcal R$ and extending~$T_0$.
By G\"odel's incompleteness theorem, there is a model of $T$ which is not
elementarily equivalent to~$\mathcal R$.

\end{document}